\documentclass[3p,10pt]{elsarticle}

\usepackage{amssymb,latexsym,amsmath,color,subfigure,amsthm}

\journal{}

\newcommand{\eps}{\varepsilon}
\newcommand{\set}[1]{\left\{#1\right\}}

\newcommand{\p}{\partial}

\newcommand{\mE}{\mathbf{E}}
\newcommand{\mH}{\mathbf{H}}

\newcommand{\mr}{\mathbf{r}}
\newcommand{\ms}[1]{\mr_{\mathrm{\tiny RX}}^{(#1)}}

\newcommand{\vt}{\boldsymbol{\theta}}

\theoremstyle{plain}
\newtheorem{thm}{Theorem}[section]

\theoremstyle{remark}
\newtheorem{rem}{Remark}[section]
\newtheorem{ex}{Example}[section]

\begin{document}

\begin{frontmatter}



\title{Direct sampling method for anomaly imaging from $S-$parameter}


\author{Won-Kwang Park}
\ead{parkwk@kookmin.ac.kr}
\address{Department of Information Security, Cryptology, and Mathematics, Kookmin University, Seoul, 02707, Korea}

\begin{abstract}
In this paper, we develop a fast imaging technique for small anomalies located in homogeneous media from $S-$parameter data measured at dipole antennas. Based on the representation of $S-$parameters when an anomaly exists, we design a direct sampling method (DSM) for imaging an anomaly and establishing a relationship between the indicator function of DSM and an infinite series of Bessel functions of integer order. Simulation results using synthetic data at $f = 1 $GHz of angular frequency are illustrated to support the identified structure of the indicator function.
\end{abstract}


\end{frontmatter}





\section{Introduction}
In this study, we consider an inverse scattering problem that determines the locations of small anomalies in a homogeneous background using $S-$parameter measurements. This study has been motivated by microwave tomography for small-target imaging, such as in the case of tumors during the early stages of breast cancer. Because of the intrinsic ill-posedness and nonlinearity of inverse scattering problems, this problem is very hard to solve; however, it is still an interesting research topic because of its relevance in human life. Many researchers have focused on various imaging techniques that are mostly based on Newton-type iteration-based techniques \cite[Table II]{CZBN}. However, the success of Newton-type based techniques is highly dependent on the initial guess, which must be close to the unknown targets. Furthermore, Newton-type based techniques have various limitations such as large computational costs, local minimizer problem, difficulty in imaging multiple anomalies, and selecting appropriate regularization. Because of this reason, developing a fast imaging technique for obtaining a good initial guess is highly required. Recently, various non-iterative techniques have been investigated, e.g., MUltiple SIgnal Classification (MUSIC) algorithm, linear sampling method, topological derivative strategy, and Kirchhoff/subspace migrations. A brief description of such techniques can be found in \cite{AGKPS,AK2,K2,P-SUB3,SZ}.

Direct sampling method (DSM) is another non-iterative technique for imaging unknown targets. Unlike the non-iterative techniques mentioned above, DSM requires either one or a small number of fields with incident directions \cite{IJZ1,IJZ2,LZ}. Furthermore, this is a considerably effective and stable algorithm. In a recent study \cite{PKLS}, the MUSIC algorithm was designed for imaging small and extended anomalies; however, DSM has not yet been designed and used to identify unknown anomalies from measured $S-$parameter data.

To address this issue, we design a DSM from $S-$parameter data collected by a small number of dipole antennas to identify the outline shape anomaly with different conductivity and relative permittivity compared to the background medium and a significantly smaller diameter than the wavelength. To investigate the feasibility of the designed DSM, we establish a relationship between the indicator function of DSM and an infinite series of Bessel functions of integer order. Subsequently, we present the simulation results that confirm the established relationship using synthetic data generated by the CST STUDIO SUITE.

The remainder of this paper is organized as follows. In Section \ref{sec:2}, we briefly introduce the DSM for imaging anomalies from $S-$parameter data. Subsequently, in Section \ref{sec:3}, we present simulation results for the synthetic data generated at $f = 1 $GHz of angular frequency, which is followed by a brief conclusion in Section \ref{sec:4}.

\section{Preliminaries}\label{sec:2}
In this section, we briefly survey the three-dimensional forward problem in which an anomaly $\mathrm{D}$ with a smooth boundary $\p\mathrm{D}$ is surrounded by $N-$different dipole antennas. For simplicity, we assume that $\mathrm{D}$ is a small ball with radius $\rho$, which is located at $\mr_\mathrm{D}$ such that
\[\mathrm{D}=\mr_\mathrm{D}+\rho\mathbf{B},\]
where $\mathbf{B}$ denotes a simply connected domain. We denote $\mr_{\mathrm{\tiny TX}}$ as the location of the transmitter, $\mr_{\mathrm{\tiny RX}}^{(n)}$ as the location of the $n-$th receiver, and $\Gamma$ as the set of receivers.
\[\Gamma=\{\mr_{\mathrm{\tiny RX}}^{(n)}:n=1,2,\cdots,N\quad\mbox{with}\quad|\mr_{\mathrm{\tiny RX}}^{(n)}|=R\}.\]
Throughout this paper, for every material and anomaly to be non-magnetic, they are classified on the basis of the value of their relative dielectic permittivity and electrical conductivity at a given angular frequency $\omega=2\pi f$. To reflect this, we set the magnetic permeability to be constant at every location such that $\mu(\mr)\equiv\mu=4\cdot10^{?7}\pi$, and we denote $\eps_\mathrm{B}$ and $\sigma_\mathrm{B}$ as the background relative permittivity and conductivity, respectively. By analogy, $\eps_\mathrm{D}$ and $\sigma_\mathrm{D}$ are respectively those of $\mathrm{D}$. Then, we introduce piecewise constant relative permittivity $\eps(\mr)$ and conductivity $\sigma(\mr)$,
\[\eps(\mr)=\left\{\begin{array}{rcl}
\eps_\mathrm{D} & \mbox{if} & \mr\in\mathrm{D},\\
\eps_\mathrm{B} & \mbox{if} & \mr\in\mathbb{R}^3\backslash\overline{\mathrm{D}},
\end{array}
\right.
\quad\mbox{and}\quad
\sigma(\mr)=\left\{\begin{array}{rcl}
\sigma_\mathrm{D} & \mbox{if} & \mr\in\mathrm{D},\\
\sigma_\mathrm{B} & \mbox{if} & \mr\in\mathbb{R}^3\backslash\overline{\mathrm{D}},
\end{array}
\right.\]
respectively. Using this, we can define the background wavenumber $k$ as
\[k=\omega^2\mu\left(\eps_\mathrm{B}+i\frac{\sigma_\mathrm{B}}{\omega}\right)=\frac{2\pi}{\lambda},\]
where $\lambda$ denotes the wavelength such that $\rho<\lambda/2$.

Let $\mE_{\mathrm{\tiny inc}}(\mr_{\mathrm{\tiny TX}},\mr)$ be the incident electric field in a homogeneous medium because of a point current density at $\mr_{\mathrm{\tiny TX}}$. Then, based on the Maxwell equation, $\mE_{\mathrm{\tiny inc}}(\mr_{\mathrm{\tiny TX}},\mr)$ satisfies
\[\nabla\times\mE_{\mathrm{\tiny inc}}(\mr_{\mathrm{\tiny TX}},\mr)=-i\omega\mu\mH(\mr_{\mathrm{\tiny TX}},\mr)\quad\mbox{and}\quad\nabla\times\mH(\mr_{\mathrm{\tiny TX}},\mr)=(\sigma_\mathrm{B}+i\omega\eps_\mathrm{B})\mE_{\mbox{\tiny inc}}(\mr_{\mathrm{\tiny TX}},\mr)\]
Analogously, let $\mE_{\mathrm{\tiny tot}}(\mr,\mr_{\mathrm{\tiny RX}}^{(n)})$ be the total field in the existence of $\mathrm{D}$ measured at $\mr_{\mathrm{\tiny RX}}^{(n)}$. Then, $\mE_{\mathrm{\tiny tot}}(\mr,\mr_{\mathrm{\tiny RX}}^{(n)})$ satisfies
\[\nabla\times\mE_{\mathrm{\tiny tot}}(\mr,\mr_{\mathrm{\tiny RX}}^{(n)})=-i\omega\mu\mH(\mr,\mr_{\mathrm{\tiny RX}}^{(n)})\quad\mbox{and}\quad\nabla\times\mH(\mr,\mr_{\mathrm{\tiny RX}}^{(n)})=(\sigma(\mr)+i\omega\eps(\mr))\mE_{\mathrm{\tiny tot}}(\mr,\mr_{\mathrm{\tiny RX}}^{(n)})\]
with transmission condition on the boundary $\p\mathrm{D}$ and the open boundary condition:
\[\lim_{|\mr|\to\infty}\mr\bigg(\nabla\times\mE_{\mathrm{\tiny tot}}(\mr,\mr_{\mathrm{\tiny RX}}^{(n)})-ik\frac{\mr}{|\mr|}\times\mE_{\mathrm{\tiny tot}}(\mr,\mr_{\mathrm{\tiny RX}}^{(n)})\bigg)=0.\]

Let $\mathrm{S}(n)$ be the $S-$parameter, which is the ratio of the reflected waves at the $n-$th receiver $\mr_{\mathrm{\tiny RX}}^{(n)}$ to the incident waves at the transmitter $\mr_{\mathrm{\tiny TX}}$. Herein, $\mathrm{S}_{\mathrm{\tiny scat}}(n)$ denotes the scattered field $S-$parameter, which is obtained by subtracting the $S-$parameters from the total and incident fields. Based on \cite{HSM2}, $\mathrm{S}_{\mathrm{\tiny scat}}(n)$ because of the existence of an anomaly, $\mathrm{D}$ can be represented as follows. This representation plays a key role in the DSM that will be designed in the next section.

\begin{equation}\label{Formula-S}
  \mathrm{S}_{\mathrm{\tiny scat}}(n)=\frac{ik^2}{4\omega\mu}\int_\mathrm{D}\chi(\mr)\mE_{\mathrm{\tiny inc}}(\mr_{\mathrm{\tiny TX}},\mr)\mE_{\mathrm{\tiny tot}}(\mr,\mr_{\mathrm{\tiny RX}}^{(n)})d\mr,\quad\chi(\mr)=\frac{\eps(\mr)-\eps_\mathrm{B}}{\eps_\mathrm{B}}+i\frac{\sigma(\mr)-\sigma_\mathrm{B}}{\omega\sigma_\mathrm{B}}.
\end{equation}

\section{Indicator function of direct sampling method: introduction and analysis}\label{sec:3}
In this section, we design an imaging algorithm based on the DSM, which uses the collected $S-$parameters $\mathrm{S}_{\mathrm{\tiny scat}}(n)$ such that $\mathbb{S}=\set{\mathrm{S}_{\mathrm{\tiny scat}}(n):n=1,2,\cdots,N}$. Because we assumed that $\mathrm{D}$ is a small ball such that $\rho<\lambda/2$, using the Born approximation, $\mathrm{S}_{\mathrm{\tiny scat}}(n)$ of (\ref{Formula-S}) can be approximated as follows:
\begin{equation}\label{Formula-S2}
\mathrm{S}_{\mathrm{\tiny scat}}(n)\approx\rho^3\frac{ik^2}{4\omega\mu}\chi(\mr_\mathrm{D})\mE_{\mathrm{\tiny inc}}(\mr_{\mathrm{\tiny TX}},\mr_\mathrm{D})\mE_{\mathrm{\tiny inc}}(\mr_\mathrm{D},\mr_{\mathrm{\tiny RX}}^{(n)}).
\end{equation}
Based on this approximation, the imaging algorithm based on the DSM can be introduced as follows; for a search point $\mr\in\Omega$, the indicator function of DSM is expressed as follows:
\begin{equation}\label{ImagingFunction}
  \mathfrak{F}_{\mathrm{DSM}}(\mr):=\frac{|\langle \mathrm{S}_{\mathrm{\tiny scat}}(n),\mE_{\mathrm{\tiny inc}}(\mr,\mr_{\mathrm{\tiny RX}}^{(n)})\rangle_{L^2(\Gamma)}|}{||\mathrm{S}_{\mathrm{\tiny scat}}(n)||_{L^2(\Gamma)}||\mE_{\mathrm{\tiny inc}}(\mr,\mr_{\mathrm{\tiny RX}}^{(n)})||_{L^2(\Gamma)}},
\end{equation}
where $\Omega$ is a search domain,
\[\langle \mathbf{F}_1(n),\mathbf{F}_2(n)\rangle_{L^2(\Gamma)}:=\sum_{n=1}^{N}\mathbf{F}_1(n)\overline{\mathbf{F}}_2(n),\quad\mbox{and}\quad||\mathbf{F}||_{L^2(\Gamma)}=\left(\langle \mathbf{F},\mathbf{F}\rangle_{L^2(\Gamma)}\right)^{1/2}.\]
Then, $\mathfrak{F}_{\mathrm{DSM}}(\mr)$ has a peak magnitude of $1$ at $\mr=\mr_\mathrm{D}$ and a small magnitude at $\mr\ne\mr_\mathrm{D}$ so that the shape of anomaly $\mathrm{D}$ can be easily identified. Following \cite{IJZ2,LZ}, the structure of $\mathfrak{F}_{\mathrm{DSM}}(\mr)$ can be represented as follows:
\[\mathfrak{F}_{\mathrm{DSM}}(\mr)\approx|J_0(k|\mr-\mr_\mathrm{D}|)|,\]
where $J_m$ is the Bessel function of the first kind of order $m$. However, this does not explain the complete phenomena that were illustrated in the simulation results in the next section; thus, further analysis is required. Through careful analysis, we can identify the structure of the indicator function as follows:

\begin{thm}[Structure of indicator function]\label{StructureDSM}
Assume that the total number of antennas $N$ is small, a sufficiently large wavenumber $k$ and search point $\mr\in\Omega$ satisfy $k|\mr-\ms{n}|\gg0.25$. Let $\vt_n=\ms{n}/|\ms{n}|=(\cos\theta_n,\sin\theta_n)$ and $\mr-\mr_\mathrm{D}=|\mr-\mr_\mathrm{D}|(\cos\phi_\mathrm{D},\sin\phi_\mathrm{D})$. Then, if $\mr$ is far from $\mr_{\mathrm{\tiny RX}}^{(n)}$,
\begin{equation}\label{ImagingFunction}
\mathfrak{F}_{\mathrm{DSM}}(\mr)=\frac{|\Phi(\mr)|}{\displaystyle\max_{\mr\in\Omega}|\Phi(\mr)|},
\end{equation}
where
\begin{equation}\label{Structure}
\Phi(\mr)=J_0(k|\mr-\mr_\mathrm{D}|)+\frac{1}{N}\sum_{n=1}^{N}\sum_{m\in\mathbb{Z}^*\backslash\set{0}}i^m J_m(k|\mr-\mr_\mathrm{D}|)e^{im(\theta_n-\phi_\mathrm{D})}.
\end{equation}
Here, $\mathbb{Z}^*=\mathbb{Z}\cup\set{-\infty,\infty}$ and $J_m$ denotes the Bessel function of integer order $m$ of the first kind.
\end{thm}
\begin{proof}
Because $k|\mr-\ms{n}|\gg0.25$, applying (\ref{Formula-S2}) and the asymptotic form of the Hankel function
\[H_0^{(2)}(k|\mr-\ms{n}|)=\frac{1+i}{4\sqrt{k\pi}}\frac{e^{ik|\ms{n}|}}{\sqrt{|\ms{n}|}}e^{-ik\vt_n\cdot\mr}+o\left(\frac{1}{\sqrt{|\ms{n}|}}\right),\]
we can observe that
\begin{align*}
\langle \mathrm{S}_{\mathrm{\tiny scat}}(n),\mE_{\mathrm{\tiny inc}}(\mr,\mr_{\mathrm{\tiny RX}}^{(n)})\rangle_{L^2(\Gamma)}&\approx\sum_{n=1}^{N}\rho^3\frac{ik^2}{4\omega\mu}\chi(\mr_\mathrm{D})\mE_{\mathrm{\tiny inc}}(\mr_{\mathrm{\tiny TX}},\mr_\mathrm{D})\mE_{\mathrm{\tiny inc}}(\mr_\mathrm{D},\mr_{\mathrm{\tiny RX}}^{(n)})\overline{\mE_{\mathrm{\tiny inc}}(\mr,\mr_{\mathrm{\tiny RX}}^{(n)})}\\
&=\frac{ik\rho^3}{32\omega\mu\pi}\chi(\mr_\mathrm{D})\mE_{\mathrm{\tiny inc}}(\mr_{\mathrm{\tiny TX}},\mr_\mathrm{D})\sum_{n=1}^{N}\frac{1}{|\ms{n}|}e^{ik\vt_n\cdot(\mr-\mr_\mathrm{D})}.
\end{align*}
Because $|\ms{n}|=R$, $\vt_n\cdot(\mr-\mr_\mathrm{D})=|\mr-\mr_\mathrm{D}|(\cos(\theta_n-\phi_\mathrm{D}),\sin(\theta_n-\phi_\mathrm{D}))$, and the following Jacobi-Anger expansion holds uniformly,
\[e^{ix\cos\theta}=J_0(x)+\sum_{m\in\mathbb{Z}^*\backslash\set{0}}i^m J_m(x)e^{im\theta},\]
we can derive
\begin{align*}
\sum_{n=1}^{N}e^{ik\vt_n\cdot(\mr-\mr_\mathrm{D})}&=\sum_{n=1}^{N}\left(J_0(k|\mr-\mr_\mathrm{D}|)+\sum_{m\in\mathbb{Z}^*\backslash\set{0}}i^m J_m(k|\mr-\mr_\mathrm{D}|)e^{im(\theta_n-\phi_\mathrm{D})}\right)\\
&=NJ_0(k|\mr-\mr_\mathrm{D}|)+\sum_{n=1}^{N}\sum_{m\in\mathbb{Z}^*\backslash\set{0}}i^m J_m(k|\mr-\mr_\mathrm{D}|)e^{im(\theta_n-\phi_\mathrm{D})}.
\end{align*}
Thus, we arrive at
\begin{multline*}
\langle \mathrm{S}_{\mathrm{\tiny scat}}(n),\mE_{\mathrm{\tiny inc}}(\mr,\mr_{\mathrm{\tiny RX}}^{(n)})\rangle_{L^2(\Gamma)}\approx\frac{ik\rho^3}{32R\omega\mu\pi}\chi(\mr_\mathrm{D})\mE_{\mathrm{\tiny inc}}(\mr_{\mathrm{\tiny TX}},\mr_\mathrm{D})\sum_{n=1}^{N}e^{ik\vt_n\cdot(\mr-\mr_\mathrm{D})}\\
=\frac{iNk\rho^3}{32R\omega\mu\pi}\chi(\mr_\mathrm{D})\mE_{\mathrm{\tiny inc}}(\mr_{\mathrm{\tiny TX}},\mr_\mathrm{D})\left(J_0(k|\mr-\mr_\mathrm{D}|)+\frac{1}{N}\sum_{n=1}^{N}\sum_{m\in\mathbb{Z}^*\backslash\set{0}}i^m J_m(k|\mr-\mr_\mathrm{D}|)e^{im(\theta_n-\phi_\mathrm{D})}\right).
\end{multline*}
Using this, we apply H{\"o}lder's inequality
\[|\langle \mathrm{S}_{\mathrm{\tiny scat}}(n),\mE_{\mathrm{\tiny inc}}(\mr,\mr_{\mathrm{\tiny RX}}^{(n)})\rangle_{L^2(\Gamma)}|\leq||\mathrm{S}_{\mathrm{\tiny scat}}(n)||_{L^2(\Gamma)}||\mE_{\mathrm{\tiny inc}}(\mr,\mr_{\mathrm{\tiny RX}}^{(n)})||_{L^2(\Gamma)},\]
to obtain (\ref{Structure}). This completes the proof.
\end{proof}

\begin{rem}\label{Remark}Based on the result of Theorem \ref{StructureDSM}, we examine some properties of the DSM.
\begin{enumerate}\renewcommand{\theenumi}{(P\arabic{enumi})}
\item Because $J_0(0) = 1$ and $J_m(0) = 0$ for all $m = 1, 2, \cdots,$ we can observe that $\mathfrak{F}_{\mathrm{DSM}}(\mr)\approx1$ at $\mr = \mr_\mathrm{D}\in\mathrm{D}$. This is the theoretical reason for which the location of $\mathrm{D}$ can be imaged using the DSM.
\item The imaging performance is highly dependent on the value of $k$ and $N$, i.e., to accurately detect the location of $\mathrm{D}$, the value of $N$ must be sufficiently large. This is the theoretical reasoning for increasing the total number of antennas to guarantee good imaging results.
\item If the value of $N$ is not sufficiently large, the right-hand side of (\ref{Structure})
\[\sum_{m\in\mathbb{Z}^*\backslash\set{0}}i^m J_m(k|\mr-\mr_\mathrm{D}|)e^{im(\theta_n-\phi_\mathrm{D})}\]
will deteriorate the imaging performance by generating large numbers of artifacts.
\item If $N$ is sufficiently large, the effect of the deteriorating term becomes negligible and $\mathfrak{F}_{\mathrm{DSM}}(\mr)$ becomes
\[\mathfrak{F}_{\mathrm{DSM}}(\mr)\approx|J_0(k|\mr-\mr_\mathrm{D}|)|.\]
This result is same as the one derived in \cite{LZ}.
\item\label{P5} If the radius of $\mathrm{D}$ is larger than $\lambda$, then it is impossible to apply Born approximation (\ref{Formula-S}). This means that the designed DSM cannot be applied to the imaging of extended targets.
\end{enumerate}
\end{rem}

\begin{rem}[Imaging of multiple anomalies]
If multiple small anomalies $\mathrm{D}_l$, $l=1,2,\cdots,L$, whose radii, permittivities, and conductivities are $\rho_l$, $\eps_l$, and $\sigma_l$, respectively, exist $\mathfrak{F}_{\mathrm{DSM}}(\mr)$ can be represented as
\[\mathfrak{F}_{\mathrm{DSM}}(\mr)=\frac{|\Phi(\mr)|}{\displaystyle\max_{\mr\in\Omega}|\Phi(\mr)|},\]
where
\[\Phi(\mr)=\sum_{l=1}^{L}\rho_l^3\left(\frac{\eps_l-\eps_\mathrm{B}}{\eps_\mathrm{B}}+i\frac{\sigma_l-\sigma_\mathrm{B}}{\omega\sigma_\mathrm{B}}\right)\left(J_0(k|\mr-\mr_\mathrm{D}|)+\frac{1}{N}\sum_{n=1}^{N}\sum_{m\in\mathbb{Z}^*\backslash\set{0}}i^m J_m(k|\mr-\mr_\mathrm{D}|)e^{im(\theta_n-\phi_\mathrm{D})}\right).\]
Based on this structure, we can observe that the imaging performance of $\mathfrak{F}_{\mathrm{DSM}}(\mr)$ is highly dependent on the values of permittivity, conductivity, size of anomalies, and the total number of dipole antennas $N$. This means that if the permittivity, conductivity, or the size of one anomaly is significantly larger than that of the others, the shape of the anomaly can be identified via the map of $\mathfrak{F}_{\mathrm{DSM}}(\mr)$. Otherwise, it will be difficult to identify the shape of the anomaly via the map of $\mathfrak{F}_{\mathrm{DSM}}(\mr)$.
\end{rem}

\section{Simulation results}\label{sec:4}
In this section, simulation results are presented to demonstrate the effectiveness of DSM and to support the mathematical structure derived in Theorem \ref{StructureDSM}. For this purpose, $N = 16$ dipole antennas were used with an applied frequency of $f = 1 $GHz. For the transducer and receivers, we set
\[\mr_{\mathrm{\tiny TX}}=0.09\mbox{m}\left(\cos\frac{3\pi}{2},\sin\frac{3\pi}{2}\right)\quad\mbox{and}\quad\mr_{\mathrm{\tiny RX}}^{(n)}=0.09\mbox{m}\left(\cos\theta_n,\sin\theta_n\right),\quad\theta_n=\frac{3\pi}{2}-\frac{2\pi(n-1)}{N}.\]
Hence, $R=|\mr_{\mathrm{\tiny RX}}^{(n)}|= 0.09 $m. The $S-$parameters $\mathrm{S}_{\mathrm{\tiny scat}}(n)$ for $n = 1,2,\cdots,N$ were generated using the CST STUDIO SUITE. The relative permittivity and conductivity of the background were set to $\eps_\mathrm{B} = 20$ and $\sigma_\mathrm{B} = 0.2 $S/m, respectively, the search domain $\Omega$ was set to be an interior of a circle with radius $0.085 $m centered at the origin, i.e., $\Omega=\set{\mr:|\mr|\leq0.085\mbox{m}}$, and the step size of $\mr$ to be of the order of $0.002 $m.

\begin{ex}[Imaging of a small anomaly]\label{Example1}
In this result, we consider the imaging of small anomalies. For this, we placed an anomaly at $(0.01\mbox{m},0.03\mbox{m})$ with a radius, relative permittivity, and conductivity of $\rho = 0.01 $m, $\eps_\mathrm{D} = 55$, and $\sigma_\mathrm{D} = 1.2 $S/m, respectively. Figure \ref{Small} shows the test configuration with the anomaly and the map of $\mathfrak{F}_{\mathrm{DSM}}(\mr)$ with an identified location of $D$. Based on these results, we detected almost the exact location of the anomaly by considering that $\mr$ satisfies $\mathfrak{F}_{\mathrm{DSM}}(\mr)\approx1$. Furthermore, because of the presence of the infinite series of Bessel functions in (\ref{Structure}), the appearance of artifacts was found to be quite different from the usual form shown in \cite{IJZ1,LZ}.
\end{ex}

\begin{figure}[h]
\begin{center}
\includegraphics[width=0.325\textwidth]{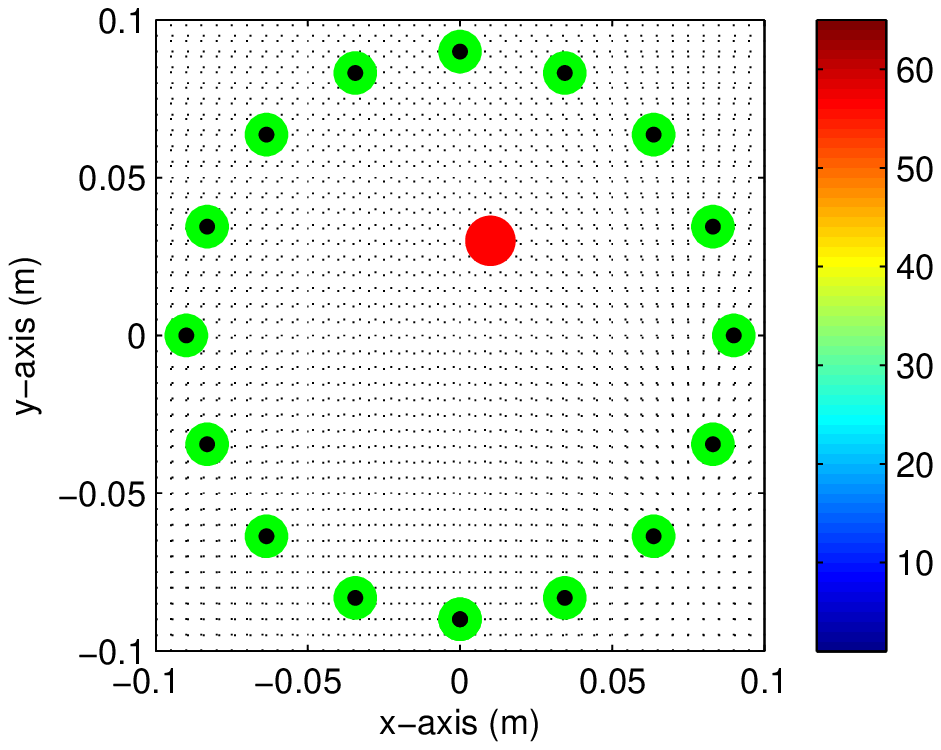}
\includegraphics[width=0.325\textwidth]{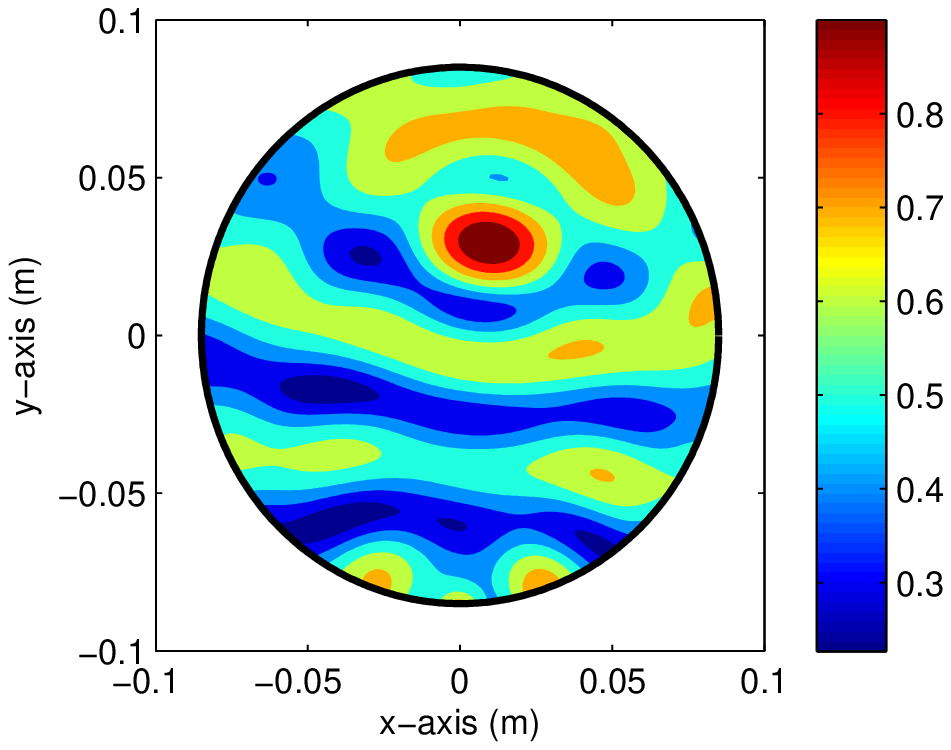}
\includegraphics[width=0.325\textwidth]{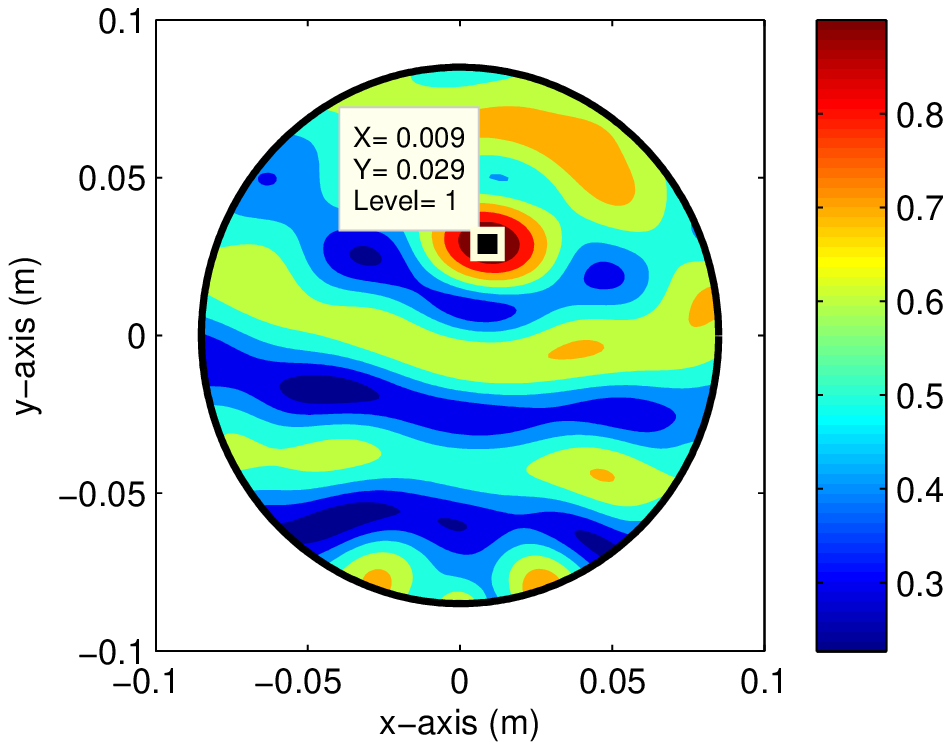}
\caption{\label{Small}Test configuration (left), map of $\mathfrak{F}_{\mathrm{DSM}}(\mr)$ (center) and the identified location of $\mathrm{D}$ (right).}
\end{center}
\end{figure}

\begin{ex}[Imaging of an extended anomaly]\label{Example2}
To examine \ref{P5} of Remark \ref{Remark}, we consider the imaging of extended anomalies. For this, we placed an anomaly at $(0.01\mbox{m},0.02\mbox{m})$ with a radius, relative permittivity, and conductivity of $\rho=0.05 $m, $\eps_\mathrm{D} = 15$, and $\sigma_\mathrm{D}=0.5 $S/m, respectively. Figure \ref{Large} shows the test configuration with the anomaly and a map of $\mathfrak{F}_{\mathrm{DSM}}(\mr)$. Based on these results, compared to the imaging of small anomalies in Example \ref{Example1}, it is impossible to recognize the shape of the anomaly. This result shows the limitation of DSM and that an improvement is necessary.
\end{ex}

\begin{figure}[h]
\begin{center}
\includegraphics[width=0.325\textwidth]{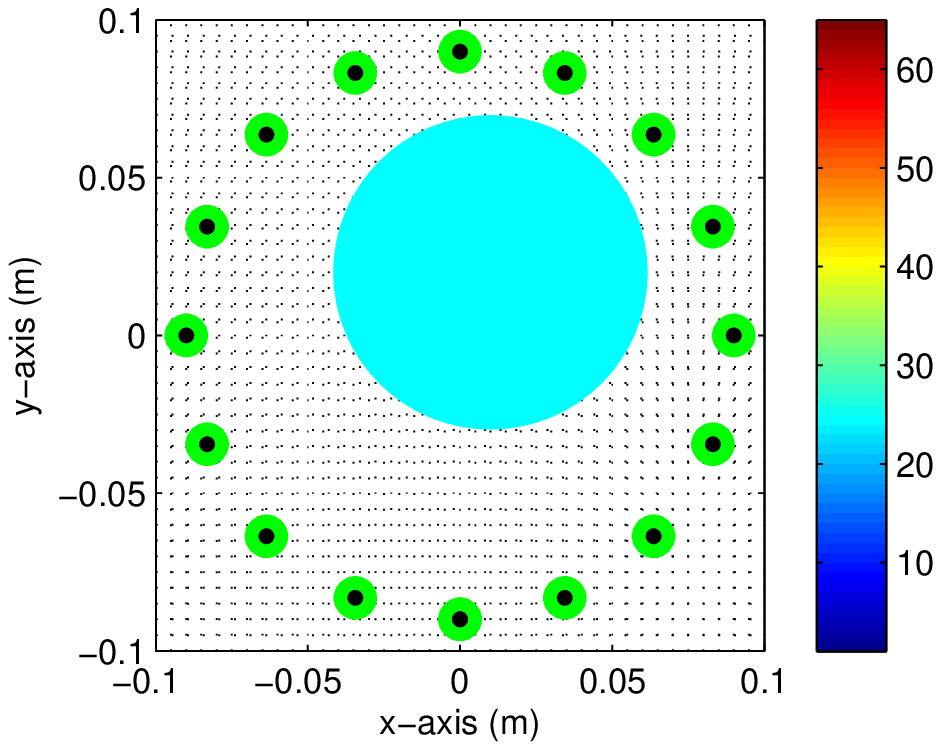}
\includegraphics[width=0.325\textwidth]{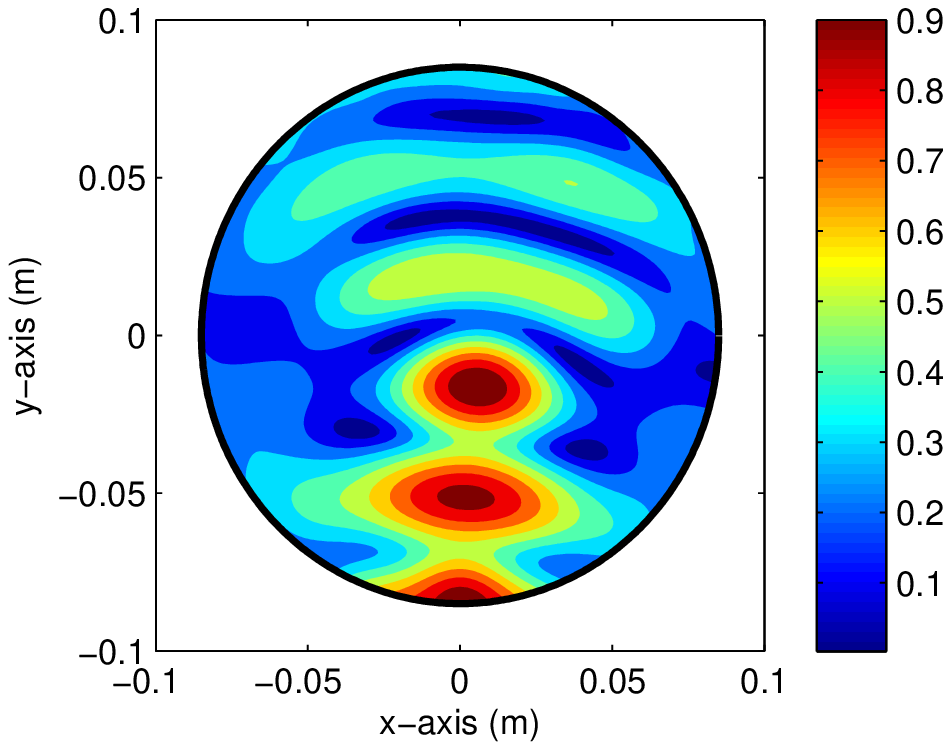}
\includegraphics[width=0.325\textwidth]{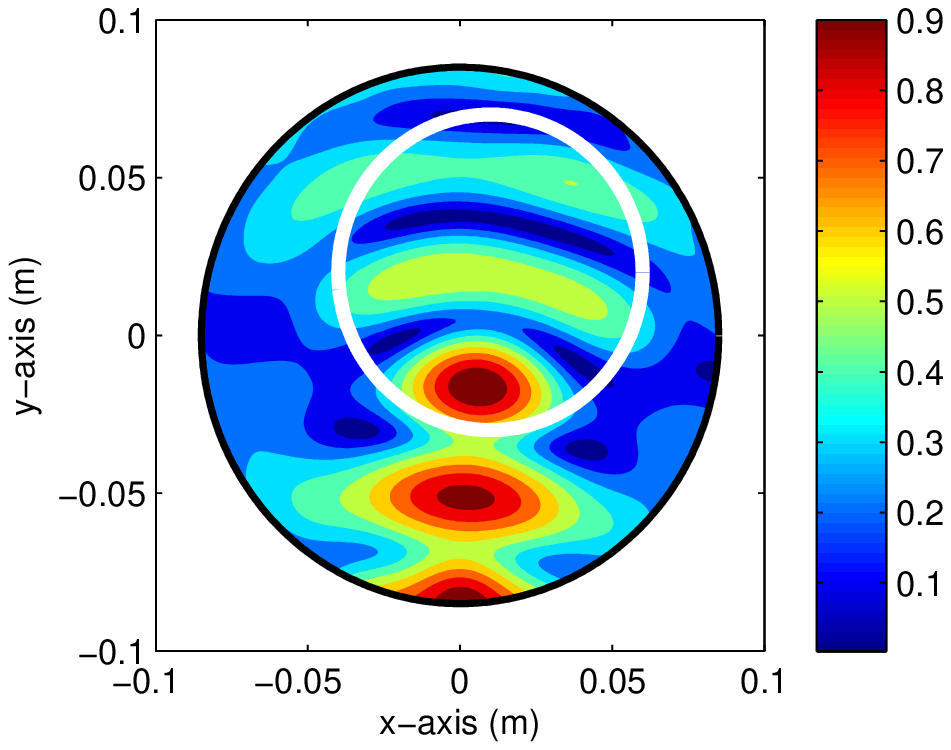}
\caption{\label{Large}Test configuration (left), map of $\mathfrak{F}_{\mathrm{DSM}}(\mr)$ (center) and the $\p\mathrm{D}$ (right).}
\end{center}
\end{figure}

\section{Conclusion}\label{sec:5}
We designed and employed DSM for fast imaging of small anomalies from $S-$parameter values. By considering the relationship between the indicator function and an infinite series of Bessel functions of integer order, certain properties of the DSM were examined. Based on the simulation results with synthetic data, we concluded that DSM is an effective algorithm for detecting small anomalies. Thus, we anticipate its development for its use in real-world applications such as breast cancer detection in biomedical imaging.

\section*{Acknowledgement}
The author is acknowledge to Kwang-Jae Lee and Seong-Ho Son for providing $S-$parameter data from CST STUDIO SUITE. This research was supported by the Basic Science Research Program through the National Research Foundation of Korea (NRF) funded by the Ministry of Education (No. NRF-2017R1D1A1A09000547).

\bibliographystyle{elsarticle-num-names}
\bibliography{../References}
\end{document}